\documentclass[12pt]{amsart}
\usepackage{amsmath,amsthm}
\usepackage{lineno}
\usepackage{indentfirst}
\usepackage{mathtools}
\usepackage [english]{babel}
\usepackage [autostyle, english = american]{csquotes}
\usepackage{graphicx}
\usepackage{tikz}
\usepackage{caption}
\MakeOuterQuote{"}

\newtheorem{thm}{Theorem}[section]
\newtheorem{prop}[thm]{Proposition}
\newtheorem{lem}[thm]{Lemma}

\newtheorem{cor}[thm]{Corollary}
\newtheorem{obs}[thm]{Observation}

\tikzstyle{vertex}=[ circle, draw=black,fill=black, inner sep=0pt, minimum size=2mm]
\tikzstyle{bluevertex}=[ circle, draw=blue,fill=blue, inner sep=0pt, minimum size=2mm]
\tikzstyle{redvertex}=[ circle, draw=red,fill=red, inner sep=0pt, minimum size=2mm]
\tikzstyle{token}=[ rectangle, draw=red,fill=none, inner sep=0pt, minimum size=3mm]

\date{(date1), and in revised form (date2).}
\thanks{The first author thanks the NSERC (Canada) Discovery Grant program.  The second author conducted this work at Thompson Rivers University and thanks the NSERC USRA program.}

\title[Broadcast Independence and Packing in Trees]{Broadcast independence and packing in certain classes of trees}
\date{}

\author{Richard C.\ Brewster}
\address{Department of Mathematics and Statistics, Thompson Rivers University, Kamloops, BC, Canada}
\email{rbrewster@tru.ca}

\author{Kiara A. McDonald}
\address{Department of Mathematics and Statistics, University ov Victoria, Victoria, BC, Canada}
\email{kiaramcdonald@uvic.ca}

\begin{document}

\begin{abstract}
Given a graph $G=(V,E)$ of diameter $d$, a \emph{broadcast} is a function $f:V(G) \to \{ 0, 1, \dots, d \}$ where $f(v)$ is at most the eccentricity of $v$. A vertex $v$ is \emph{broadcasting} if $f(v)>0$ and a vertex $u$ \emph{hears} $v$ if $d(u,v) \leq f(v)$.  A broadcast is \emph{independent} if no broadcasting vertex hears another vertex and is a \emph{packing} if no vertex hears more than one vertex. The \emph{weight} of $f$ is $\sum_{v \in V} f(v)$. We find the maximum weight independent and packing broadcasts for perfect $k$-ary trees, spiders, and double spiders as a partial answer to a question posed by Ahmane et al.
\end{abstract}

\maketitle
\section{Introduction}

Let $G = (V, E)$ be a finite, simple undirected graph. We let
$\mathrm{diam}(G)$ denote the \emph{diameter} of $G$ (the maximum distance between two vertices in $G$) and $\mathrm{ecc}_G(v)$ denote the \emph{eccentricity} of a vertex $v \in V(G)$ (the maximum distance of a vertex from $v$). A \emph{broadcast} on $G$ is a function $f: V(G) \rightarrow \{0, 1, \dots, \mathrm{diam}(G)\}$, where $f(v) \leq \mathrm{ecc}_G(v)$ for every vertex $v \in V(G)$. The \emph{cost} or \emph{weight} of a broadcast $f$ is defined as $f(V) := \sum_{v \in V(G)}{f(v)}$. A vertex $v$ is \emph{broadcasting} if $f(v) > 0$ and $V_f^+$ is the set of broadcasting vertices in $f$. A vertex $u$ \emph{hears} the broadcast if $d(u,v) \leq f(v)$ for some $v \in V_f^+$. (Clearly, a broadcasting vertex hears itself.) The set of vertices which $u$ hears is denoted by $H(u) = \{v \in V_f^+| d(u,v) \leq f(v) \}$. A vertex $u$ is \emph{f-dominated}, or simply \emph{dominated} when $f$ is clear from context, if $|H(u)| \geq 1$.
A broadcast $f$ is \emph{dominating} if every vertex in $G$ is $f$-dominated.  That is, $\bigcup_{v \in V_f^+} N_{f(v)}[v] = V(G)$ where $N_{k}[v] = \{ u | d(u,v) \leq k \}$ is the \emph{ball of radius $k$ centred at $v$}.  Note that $N_{f(v)}[v]$ is precisely the set of vertices that hear $v$.

A broadcast $f$ is \emph{independent} if no broadcast vertex hears another vertex; that is, for every $v \in V_f^+$, $|H(v)| = 1$. The \emph{broadcast independence number} $\alpha_b(G)$ of $G$ is the maximum weight of an independent broadcast on $G$. We note this definition was introduced in \cite{Dunbar2006, Erwin2001}. A more refined study of broadcast independence is in~\cite{Mynhardt2021, Neilson2019} where this is called \emph{hearing-independence}.  Note the value of $\alpha_b(G)$ for a graph $G$ consisting of a single vertex is $0$ as the diameter of $G$ is $0$, although one might argue is natural to broadcast with power $1$ in this special case.  To avoid this issue we will restrict our attention to connected graphs on at least $2$ vertices for the remainder of the paper.

If we require the stronger condition that each vertex (broadcasting or not) can hear at most one vertex, then we arrive at the following notion. A broadcast $f$ is a \emph{packing broadcast} if for every $u \in V(G)$, $|H(u)| \leq 1$. The maximum weight of a packing broadcast on $G$ is the \emph{broadcast packing number}, denoted by $P_b(G)$.

Certain broadcast parameters have been calculated for special classes of graphs~\cite{Ahmane2018, Lobsters2024, Bouchemakh2014, Bouchouika2020, Circulant2024}.  Ahmane et al.~\cite{Ahmane2018} examine broadcast independence in caterpillars and pose the question ``Can we determine the broadcast independence number of other subclasses of trees? In particular, what about $k$-ary trees?''. In this paper, we answer this question for spiders and perfect $k$-ary trees.

In Section~\ref{sec:indeptrees}, we determine an explicit formula for the broadcast independence number of perfect $k$-ary trees. In Section~\ref{sec:spiders}, we determine $\alpha_b$ for spiders.  Finally, in Section~\ref{packing}, we determine optimal broadcast packings and their duals which we call \emph{multicovers} for perfect $k$-ary trees, spiders, caterpillars, and double spiders.  The equality of the two dual parameters is immediate from~\cite{Farber1984, Lubiw1987}, and so our contribution is to provide a structural description of the primal and dual solutions.

In general we follow the notation of~\cite{west}.

\section{Broadcast independence in perfect $k$-ary trees}\label{sec:indeptrees}

A \emph{perfect $k$-ary tree} is a rooted tree in which every vertex has exactly $k$ children and all leaves are at the same height. A \emph{perfect binary tree} is a $k$-ary tree such that $k=2$. In this paper, our results differ for the cases $k=2$ and $k \geq 3$.  Therefore we adopt the convention that \emph{perfect $k$-ary tree} is used when $k \geq 3$ (unless explicitly stated otherwise) and \emph{perfect binary tree} is used when $k=2$. A perfect $k$-ary tree of height $h$ is denoted $T_h^k$, and a perfect binary tree of height $h$ is denoted $T_h$. In this section we give exact values for the broadcast independence number of perfect binary trees and $k$-ary trees.

We begin with some basic results extending ideas from~\cite{Bouchemakh2014}.  Given a broadcast $f$ on $G$, we denote the restriction of $f$ to a subgraph $H$ by $f\vert_H$.  Recall, a subgraph $H$ is \emph{isometric} in $G$ if for all $x, y \in V(H)$, $d_H(x,y) = d_G(x,y)$. Clearly, an isometric subgraph must be induced.

\begin{obs}\label{obs:isometric}
Let $f$ be an independent broadcast on a graph $G$, and let $H$ be an isometric subgraph of $G$. If $\lvert{V^+_{f\vert_H}}\rvert \geq 2$, then $f\vert_H$ is an independent broadcast on $H$.
\end{obs}

\begin{proof}
Let $x$ and $y$ be distinct broadcast vertices in $H$.  Then 
\[
f(x) < d_G(x,y) = d_H(x,y) \leq e_H(x)
\]
implying $y$ does not hear $x$, and $f$ is a well-defined broadcast on $H$. The result follows.
\end{proof}

The following result is the key tool allowing us to use recursion on subtrees.

\begin{lem}\label{lem:isocollection}
Let $f$ be an independent broadcast on a graph G. Let $H_1, H_2, \dots, H_k$ be a collection of pairwise vertex disjoint, isometric subgraphs of $G$ such that $\lvert{V^+_{f\vert_{H_i}}}\rvert \geq 2$ for each $i$.
Then 
\[
f(V(G)) \leq \sum_{i=1}^{k}{\alpha_{b}(H_i)} + \sum_{v \in V(G)\backslash V(H)} f(v).
\]
\end{lem}
\begin{proof}
From Observation~\ref{obs:isometric}, we know $f\vert_{H_i}$ is an independent broadcast on $H_i$ for each $i$ giving $f(V(H_i)) \leq \alpha_{b}(H_i)$. The result is immediate.
\end{proof}

An idea used throughout the paper is the following.  If $f$ is an $\alpha_b$-broadcast on $G$, then $f$ is dominating.  If some vertex $v$ does not hear the broadcast, then we may set $f(v)=1$ (and leave all other vertices unchanged) to obtain an independent broadcast of heavier weight. Using this we prove the following, which states in our setting ``leaves hear leaves''.

\begin{lem}\label{lem:leafhearleaf}
Let $T$ be a perfect binary or $k$-ary tree and let $f$ be an $\alpha_b$-broadcast.  Then each leaf of $T$ is $f$-dominated by a leaf. 
\end{lem}

\begin{proof}
Suppose to the contrary that there is a leaf $l$ which is $f$-dominated by a non-leaf vertex $v$, and suppose $l$ and $v$ are chosen so that $d(l,v)$ is minimum with respect to this property.  In particular, $l$ is a leaf in the subtree rooted at $v$ (since $T$ is perfect) and $v$ dominates the entire subtree.
Let $i = d(l,v)$. Define a broadcast $g$ by
\[ g(x)=
    \begin{dcases}
        f(v) + i & x=l \\
        0 & x=v \\
        f(x) & \mbox{otherwise}. \\
    \end{dcases}
\]
Let $z$ be any $f$-broadcast vertex distinct from $v$.  Since $l$ is in the subtree rooted at $v$ and $f$ is independent, $z$ is not in the subtree rooted at $v$.  Thus, the unique $l-z$ path in $T$  passes through $v$.
Hence, $d(l,z)=d(l,v)+d(v,z) \geq i+f(v)+1 > g(l)$. So $g$ is independent, with $g(V)>f(V)$. This contradicts our choice of $f$.
\end{proof}

We note that the above lemma does not hold in general trees (a perfect binary tree of height $8$ with a leaf added to a grandchild of the root provides an counterexample). This is in contrast to optimal boundary independent broadcasts~\cite{Neilson2019} where leaves only hear leaves for any tree.

\subsection{Perfect Binary Trees}

We now determine $\alpha_b$ for perfect binary trees.

\begin{obs}\label{obs:leafcount}
Let $T$ be a perfect binary tree and $f$ be a broadcast. Suppose $l$ is a leaf such that $f(l)>0$.
Then $l$ dominates $2^{\lfloor{\frac{f(l)}{2}}\rfloor}$ leaves.
\end{obs}

\begin{proof}
Let $v$ the ancestor of $l$ at distance $\lfloor{\frac{f(l)}{2}}\rfloor$.  It is easy to see $l$ dominates precisely the leaves in the subtree rooted at $v$.
\end{proof}

\begin{lem}\label{lem:leafpower}
Let $T$ be a perfect binary tree of height at least $2$ and suppose $f$ is an $\alpha_b$-broadcast. If $l$ is a broadcasting leaf, then $f(l)= 3$.
\end{lem}

\begin{proof}
Suppose to the contrary that there is some leaf $l$ such that $f(l) \geq 4$. Using Observation~\ref{obs:leafcount}, let $V_{f(l)} = \{v_1,v_2, \ldots, v_{2^{\lfloor{\frac{f(l)}{2}\rfloor}}}\}$ be the set of leaves which $l$ dominates ordered by non-decreasing distance from $l$, in particular $v_1=l$. Define a broadcast $g$ as
\[ g(x)=
    \begin{dcases}
        3 & x=v_i \in V_{f(l)}, i \textrm{ is odd} \\
        0 & x=v_i \in V_{f(l)}, i \textrm{ is even} \\
        f(x) & \textrm{otherwise}. \\
    \end{dcases}
\]
Note that $v_i$ and $v_j$ are at distance at least $4$ from each other unless $j = i \pm 1$.  This implies $g$ is independent.  As $3\cdot 2^{\lfloor{\frac{f(l)}{2}}\rfloor -1} > f(l)$ (for $f(l) \geq 4$), we have our desired contradiction.

Now suppose some leaf $l$ satisfies $f(l)=1$. Let $l'$ be the sibling of $l$.  By Lemma~\ref{lem:leafhearleaf}, $f(l')=1$. Let $z$ be the ancestor at distance $2$ from $l$ and let $l''$ be a leaf at distance $4$ from $l$ in the subtree rooted at $z$.  If $f(z) > 0$, then $f(z)=1$ and $l''$ together with its sibling have a total power of $2$. Define a broadcast $g$ by 
\[ g(x)=
    \begin{dcases}
        3 & x \in \{l, l''\} \\
        0 & x \mbox{ in the subtree rooted at $z$, $x \neq l, l''$ } \\
        f(x) & \textrm{otherwise}. \\
    \end{dcases}
\] 
It is easy to see $g(V) > f(V)$. The set of vertices that were dominated by $z$ and its descendants is precisely the set now dominated by $l$ and $l''$ showing $g$ is independent. 

Otherwise, $f(z) = 0$.  In this case, we set $g(l) = 2, g(l')=0$ and leave all other values unchanged giving $f(V)=g(V)$.  Hence, we may assume $f$ takes the value $2$ or $3$ on all leaves.  We now prove that there can be no leaf with power $2$.

Suppose to the contrary $f(l)=2$ for some leaf $l$. If $T$ has height $2$, then the broadcast $g$ obtained by setting $g(l)=3$ and leaving other values unchanged is an independent broadcast of greater weight. Hence, $T$ has height at least $3$ and we let $z$ be the ancestor of $l$ at distance $3$. The subtree rooted at $z$ is $T_z$ as depicted in Figure~\ref{fig:Tz}. If $f(z)=0$, then again we can increase the power at $l$ to $3$ and obtain an independent broadcast of greater weight.  

Consequently, $1 \leq f(z) \leq 2$.
By the above results and using the fact that $f$ is both independent and dominating (as any $\alpha_b$-broadcast must be), the leaves in $T_z$ must broadcast with power $2$ or $0$. Moreover, without loss of generality we assume $f(v_i) = 2$ for $i=1, 3, 5, 7$ and all other leaves are non-broadcasting (as depicted in Figure~\ref{fig:Tz}).
\begin{center}
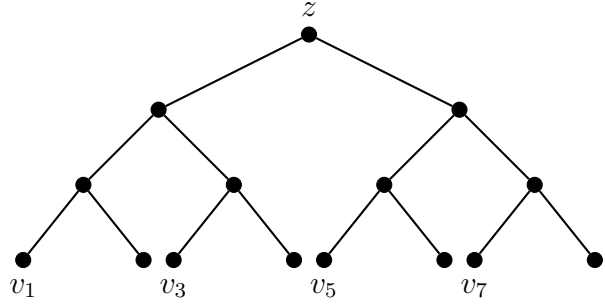

\begin{tikzpicture}
  \node[vertex,label={90:$z$}] (v0) at (4,3) {};
  \node[vertex] (v10) at (2,2) {};
  \node[vertex] (v11) at (6,2) {};
  \node[vertex] (v20) at (1,1) {};
  \node[vertex] (v21) at (3,1) {};
  \node[vertex] (v22) at (5,1) {};
  \node[vertex] (v23) at (7,1) {};

  \node[vertex,label={270:$v_1$}] (v30) at (0.2,0) {};
  \node[vertex] (v31) at (1.8,0) {};
  \node[vertex,label={270:$v_3$}] (v32) at (2.2,0) {};
  \node[vertex] (v33) at (3.8,0) {};

  \node[vertex,label={270:$v_5$}] (v34) at (4.2,0) {};
  \node[vertex] (v35) at (5.8,0) {};
  \node[vertex,label={270:$v_7$}] (v36) at (6.2,0) {};
  \node[vertex] (v37) at (7.8,0) {};

  \draw[thick] (v30)--(v20)--(v31) (v32)--(v21)--(v33)
  (v34)--(v22)--(v35) (v36)--(v23)--(v37);

  \draw[thick] (v20)--(v10)--(v21) (v22)--(v11)--(v23); 
  
  \draw[thick] (v10)--(v0)--(v11);
\end{tikzpicture}
    
    \captionof{figure}{The subtree $T_z$}\label{fig:Tz}
\end{center}
The weight of $f$ on the subtree $T_z$ is $9$ or $10$. The broadcast $g$ defined by
\[ g(x)=
    \begin{dcases}
        3 & x\in \{ v_1, v_3, v_5, v_7 \} \\
        0 & x\in \{z\} \\
        f(x) & \textrm{otherwise} \\
    \end{dcases}
\]
is independent as $N_{f(v_i)}[v_i]$ contained within $T_z$ for each $i=1, 3, 5, 7$. Also, $g(V) > f(V)$, contradicting our choice of $f$.

Therefore, in an $\alpha_b$-broadcast $f$ of a perfect binary tree of height $h$, all broadcasting leaves have power $3$.
\end{proof}

We remark that broadcasting with power $3$ from each $v_i$ and with power $0$ from the corresponding siblings is a pattern repeated throughout this section.  We will refer to such a broadcast as \emph{broadcasting from half the leaves}.

\begin{lem}\label{lem:atmost4}
Let $f$ be an $\alpha_b$-broadcast on a perfect binary tree $T$.  Then $f(v) \leq 4$ for all $v \in V(T)$. 
\end{lem}

\begin{proof}
Suppose to the contrary that there is some $v$ such that $f(v)=p\geq 5$. By Lemma~\ref{lem:leafpower}, we know $v$ cannot be a leaf. Let $T'$ be the subtree rooted at $v$. At distance $p-3$ from $v$, there are $2^{p-3}$ vertices in $T'$. Let $S = v_1,v_2, \ldots, v_{2^{p-3}}$ be these $2^{p-3}$ vertices in the natural order from left to right. Define a broadcast $g$ as
\[ g(x)=
    \begin{dcases}
        3 & x=v_i\in S, \textrm{ $i$ is odd} \\
        0 & x=v \\
        f(x) & \textrm{otherwise} \\
    \end{dcases}
\]
Since $f$ is independent, the distance from $S$ to a broadcast vertex in $T'$ (other than $v$) is at least 4, showing $g$ is independent. Further, $g(V)>f(V)$ as $3\cdot 2^{p-4} > p$, $p\geq 5$. This contradicts our choice of $f$.
\end{proof}

Below we give a formula for $\alpha_b(T_h)$ for $h \geq 2$. We remark the exceptional case $h=1$ has $\alpha_b(T_1)=2$.

\begin{thm}
Let $T_h$ be a perfect binary tree of height $h=4m+r \geq 2$, $1 \leq r \leq 4$. Then
\[
\alpha_b(T_h) = 3 \cdot 2^{r-1} \frac{16^{m+1}-1}{15}+b_r
\]
where $b_1=1, b_2=b_3=0, b_4=3$
\end{thm}

\begin{proof}
By Lemmas~\ref{lem:leafhearleaf} and~\ref{lem:leafpower}, the result is immediate for $2 \leq h \leq 3$.

For $h=4$, we note $T_4$ is composed of two $T_3$ trees each connected to the root vertex, which we will call $v$. Let $f$ be an $\alpha_b$-broadcast. As no leaf hears $v$, $f(v) \leq 3$. By Lemma~\ref{lem:isocollection}, $f(V(T_4)) = \alpha_{b}(T_4) \leq 2\cdot \alpha_{b}(T_3) + f(v) \leq 2\cdot 12 +3 = 27$. This upper bound is achieved by broadcasting with power $3$ from half the leaves and power $3$ from $v$, which is an independent broadcast. 

Similar to $h=4$, let $v$ be the root of $T_5$ and observe the four grandchildren of $v$ are roots of a copy of $T_3$. Let $f$ be an $\alpha_b$-broadcast and let $T'$ be the copy of $T_1$ rooted at $v$. By Lemma~\ref{lem:isocollection}, $\alpha_{b}(T_5) \leq 4\cdot \alpha_{b}(T_3) + \sum_{u \in V(T')} f(u) \leq 4 \cdot 12 + 4 = 52$. (In $T'$, $f$ either broadcasts from a single vertex with power at most $4$ by Lemma~\ref{lem:atmost4}, or from two vertices with power $1$.) The upper bound is achieved by broadcasting with power $3$ from half the leaves and power $4$ from the root $v$. 

Thus suppose $h \geq 6$. Let $f$ be an $\alpha_b$-broadcast on $T_h$. There are $2^{h-3}$ vertices at level $h-3$ each of which is the root of a perfect binary tree of height $3$.  Call these subtrees $H_1, H_2, \dots, H_{2^{h-3}}$. Then $T_{h-4} = T_h \setminus \{ H_1, H_2, \dots, H_{2^{h-3}}\}$ is a perfect binary tree of height $h-4$. By Lemmas~\ref{lem:leafhearleaf} and~\ref{lem:leafpower}, we know $f(l)=3$ for exactly four leaves of each $H_i$.  

We claim there are at least two broadcast vertices in $T_{h-4}$. As $f$ must be dominating there is at least one broadcast vertex.  (No vertex in $T_{h-4}$ is $f$-dominated by a leaf from some $H_i$.) Since $h-4 \geq 2$, there are two leaves, say $l_1$ and $l_2$, in $T_{h-4}$ at distance $4$ from each other.  Moreover, they are at distance at least $4$ from the leaves of the original tree $T_h$.  Hence the broadcast defined by
\[
g(x)=
    \begin{dcases}
        3 & x \in \{ l_1, l_2 \} \\
        f(x) & x \in V(H_i), 1 \leq i \leq 2^{h-3} \\
        0 & \textrm{otherwise} \\
    \end{dcases}
\]
is an independent broadcast.  Thus, $f(V(T_{h-4})) \geq 6$.  By Lemma~\ref{lem:atmost4}, $f$ must broadcast on at least two vertices in $T_{h-4}$. 

The collection $H_1, H_2, \dots, H_{2^{h-3}}, T_{h-4}$ is a partition of the entire tree $T_h$ into vertex disjoint, isometric subgraphs with $f$ broadcasting on at least two vertices in each subgraph.  By Lemma~\ref{lem:isocollection},
\[
\alpha_b(T_h) \leq 2^{h-3} \cdot \alpha_b(T_3) + \alpha_b(T_{h-4})
\]
Moreover, using an $\alpha_b$-broadcast on each subgraph produces an independent broadcast, as we broadcast at leaves in each subgraph with power $3$ and the leaves in different subgraphs of the partition are distance at least $4$ apart.  That is, 
\begin{eqnarray*}
\alpha_b(T_h) &=& 2^{h-3} \cdot \alpha(T_3) + \alpha(T_{h-4}) \\
              &=& 2^{h-3} \cdot 12 + 3 \cdot 2^{r-1} \frac{16^{m}-1}{15}+b_r \\
              &=& 3 \cdot 2^{r-1} \left( 2^{h-r} + \frac{16^m-1}{15} \right)+b_r \\
              &=& 3 \cdot 2^{r-1} \frac{16^{m+1}-1}{15}+b_r
\end{eqnarray*}
\end{proof}

\subsection{Perfect $k$-ary Trees}

The development for perfect $k$-ary trees (recall we require $k\geq 3$), is similar to perfect binary trees. The difference is in this setting we broadcast with power $1$ from every leaf.  In the recursive step we will reduce the height by 2 (versus 4 above).  We begin with a straightforward observation.

\begin{obs}\label{obs:leafcountk}
Let $T$ be a perfect $k$-ary tree and $f$ be a broadcast. Suppose $l$ is a leaf such that $f(l)>0$.
Then $l$ dominates $k^{\lfloor{\frac{f(l)}{2}}\rfloor}$ leaves.
\end{obs}

\begin{lem}\label{lem:fpowerk}
Let $T$ be a perfect $k$-ary tree and $f$ be an $\alpha_b(T)$-broadcast. Then $f(v) \leq 1$ for all $v \in V(T)$ with $f(l) = 1$ for each leaf of $T$.
\end{lem}

\begin{proof}
We begin with the leaves. Suppose to the contrary $f(l) = p > 1$ for some leaf $l$.  Then $f$ dominates $k^{\lfloor{\frac{p}{2}}\rfloor}$ leaves. Create a new broadcast $g$ by setting $g(l')=1$ for each leaf that hears $l$ and $g(x)=f(x)$ otherwise.  Note, $k^{\lfloor{\frac{p}{2}}\rfloor} \geq p$ for $k \geq 3, p \geq 2$, with a strict inequality except for the case $k=3$ and $p=3$.  In this case we also set $g(u)=1$ for the grandparent $u$ of $l$.  As $g(V) > f(V)$, we obtain a contradiction.

On the other hand, an $\alpha_b$-broadcast must be dominating. By Lemma ~\ref{lem:leafhearleaf}, each leaf can only hear itself; hence, $f(l)=1$ for all leaves.

Finally, let $v$ be an internal vertex of $T$ and suppose to the contrary that $f(v) = p \geq 2$.  Over all such internal vertices, choose $v$ so that no internal vertex in the subtree rooted at $v$ broadcast with power greater than $1$. Let $T'$ be the subtree rooted at $v$. Since $f$ is independent, by Lemma~\ref{lem:leafhearleaf}, $T'$ must have height at least $p+1$. Let $S$ be the $k^{p-1}$ vertices at height $p-1$ in $T'$. Define a broadcast $g$ as
\[ g(x)=
    \begin{dcases}
        1 & x \in S \\
        0 & x=v \\
        f(x) & otherwise \\
    \end{dcases}
\]
We note $g$ is independent and $g(V) > f(V)$ as $k^{p-1} >p$ for $k \geq 3$ and $p \geq 2$. This contradicts our choice of $f$.
\end{proof}

A direct consequence of Lemma~\ref{lem:fpowerk} is if $f$ is an $\alpha_b$-broadcast in $T^k_h$, then the set of broadcast vertices form an independent set giving $\alpha(T_h^k) = \alpha_b(T_k^h)$. From this observation and the formula for $\alpha(T_h^k)$ we obtain the following.

\begin{thm}
Let $T_{h}^{k}$ be a perfect $k$-ary tree of height $h$.  Then $\alpha_b(T_h^k) = \alpha(T_h^k)$.
In particular,
\[ \alpha_{b}(T_{h}^{k})=
    \begin{dcases}
        1+k^2+\cdots+k^{2m} = \frac{{(k^2)}^{m+1}-1}{k^{2}-1} & \mbox{if } h=2m \\
        k + k^3 + \cdots + k^{2m+1} = k \cdot \frac{{(k^2)}^{m+1}-1}{k^{2}-1} & \mbox{if } h=2m+1 \\
    \end{dcases}
\]
\end{thm}

\section{Broadcast independence in spiders}\label{sec:spiders}
\subsection{Preliminary Results}

We now turn our attention another class of trees. A \emph{spider} is a tree obtained by subdividing $K_{1,k}$ where $k \geq 3$.  That is, there are $k$ leaves, one vertex $u$ such that $d(u) = k$, and the remaining vertices have degree $2$. Label the leaves $\{l_1, \dots, l_k \}$ in non-decreasing order of distance to $u$ and denoted by $d_i$ the distance from $l_i$ to $u$. We can thus denote the spider by $S(d_1, \dots, d_k)$. A path from a leaf to $u$ is called a \emph{branch}, and $u$ is called the \emph{branch vertex}.

Our first result says in an $\alpha_b$-broadcast of a spider, only leaves broadcast.

\begin{lem}\label{lem:spiderOnlyleaves}
Let $f$ be an $\alpha_b$-broadcast in a spider $T$. 
If $f(v) > 0$, then $v$ is a leaf.
\end{lem}

\begin{proof}
Suppose to the contrary that $f$ is an $\alpha_b$-broadcast and $f(v) > 0$ for some internal vertex $v$.  Choose $v$ such that there is a leaf $l$ where (i) $l$ and $v$ belong to the same branch and (ii) the unique $(l,v)$-path does not contain any broadcast vertices other than $l$ or $v$.  Define
\[ g(x)=
    \begin{dcases}
        f(v) + f(l)+1 & \mbox{ if } x=l \\
        0 & \mbox{ if } x=v \\
        f(x) & \mbox{ otherwise.} \\
    \end{dcases}
\]
(Note $f(l) = 0$ is possible.)
Let $f(w) > 0$ where $w$ is not $l$ or $v$.  By our choice of $v$, the unique $(l,w)$-path contains $v$ as an internal vertex.  By the independence of $f$, $v$ does not hear $l$ and $w$ does not hear $v$, giving 
\[
d(l,w) = d(l,v)+d(v,w) > f(l) + f(v) + 1 = g(l).
\]
Hence $g$ is independent and $f(V) < g(V)$, a contradiction.
\end{proof}

\begin{lem}\label{lem:spiderleavesbroad}
Let $T$ be a spider with the branch vertex $u$ and let $L = \{l_1, \dots, l_k\}$ be the set of leaves of $T$, ordered by non-decreasing distance from $u$. Then there is an $\alpha_b$-broadcast $f$ and some $t$, $1 \leq t \leq k$, such that $f(l_i) > 0$ for $t \leq i \leq k$ and otherwise $f(v) = 0$ for all other vertices $v$. Moreover, if two leaves have the same distance to $u$ and both are broadcasting, then all leaves at that distance from $u$ are broadcasting.
\end{lem}

\begin{proof}
Let $f$ be an $\alpha_b$-broadcast. By Lemma~\ref{lem:spiderOnlyleaves}, only leaves broadcast. If $f$ does not have the form described in the statement, then there are leaves $l_i$ and $l_j$ such that $i < j$, $f(l_i)>0$, $f(l_j) = 0$. 

If $d(l_i,u) < d(l_j,u)$, we can set $f(l_j) = f(l_i)+1$ and $f(l_i)=0$, maintaining the independence of $f$ while increasing the total weight, a contradiction. Thus $d(l_i,u) = d(l_j,u)$. We set $f(l_j) = f(l_i)$ and $f(l_i)=0$. Continuing in this way, we can change $f$ to have the desired form.

Now, suppose there are leaves $l_i, l_{i+1}$ and $l_m$ such that $d(l_i,u) = d(l_{i+1},u) = d(l_m, u)$. Further, suppose $l_i, l_{i+1} \in V_f^{+}$. By independence, $l_i$ does not hear $l_{i+1}$, so $l_m$ does not hear $l_{i+1}$ or $l_i$. As $l_m$ must hear some leaf, it is easy to see the only possibility is that $l_m$ hears itself.

\end{proof}

\subsection{Determining $\alpha_b$ in Spiders}

Let $S(d_1, \dots, d_k)$ be a spider. There is an $\alpha_b$-broadcast of the form described by Lemma ~\ref{lem:spiderleavesbroad}. The following proposition allows us to find such an optimal broadcast.
\begin{prop}\label{prop:spiderbound}
    Let $t$ be a fixed integer $1 \leq t \leq k$ and let $f_t$ be a broadcast on $S(d_1, \dots, d_k)$ such that $V_{f_t}^+ = \{ l_t, \dots, l_k \}$.  Then
    \[
    f_t(V) \leq (d_t+d_{t+1}-1) + \sum_{j=t+1}^k (d_t + d_j -1).
    \]
\end{prop}

\begin{proof}
    The distance between broadcasting vertices $l_t, l_j$, $t < j$, is given by $d(l_t,l_j) = d_t + d_j$. Thus, $f_t(l_j) \leq d_t + d_j - 1$. As the closest broadcast leaf to $l_t$ is $l_{t+1}$, $f_t(l_t) \leq d_t + d_{t+1}-1$.
\end{proof}
By choosing $f_t$ to have each leaf broadcast at its maximum allowed power, the computation of $\alpha_b(S(d_1, \dots, d_k))$ is now immediate.
\begin{cor}
    For the spider $S(d_1, \dots, d_k)$,
    \[
    \alpha_b = \max_{1 \leq t \leq k}  \left[ (d_t+d_{t+1}-1) + \sum_{j=t+1}^k (d_t + d_j -1) \right].
    \]
\end{cor}

\section{Broadcast packings and multicovers}\label{packing}

As stated in the introduction, a broadcast is a collection of neighbourhood balls.  A dominating broadcast is a covering, and an independent broadcast is a packing where each broadcast vertex hears only itself.  If we strengthen the constraint that no vertex hears more than one broadcasting vertex, we have a packing.
The concepts of packing and covering are well studied.  Indeed the optimization version of these problems can be expressed (as dual) linear programs.  See~\cite{Farber1984,Lubiw1987} for early work, and~\cite{Brewster2013} for the specific application to broadcast domination where the dual notion of \emph{multipacking} is defined.  Here we express $P_b(G)$ using linear programming and extract its dual which we call the \emph{multicover} problem.

Given a graph $G=(V,E)$ for each vertex $v \in V$ and for each $k$, $1 \leq k \leq \mathrm{ecc}_G(v)$, we introduce a binary variable $x_{v,k}$.  If $x_{v,k}=1$, then the ball of radius $k$ about $v$ is in the packing. The broadcast packing number is the maximum weight broadcast subject to the constraint that each vertex hears at most one broadcast vertex.  In symbols, 

\begin{equation}\label{eqn:Pb}
\begin{matrix}
P_b(G) & = & \displaystyle \max \sum_{v \in V} \sum_{k=1}^{\mathrm{ecc}_G(v)} k \cdot x_{v,k} & \\[20pt]
  & s.t. & \displaystyle \sum_{d(u,v) \leq k} x_{v,k} \leq 1 & \mbox{ for each } u \in V \\[20pt]
   & & x_{v,k} \in \{ 0, 1 \}
\end{matrix}   
\end{equation}

To define the dual problem, \emph{multicover}, we introduce a binary variable $y_u$ for each $u \in V$. When $y_u = 1$ we will say we have placed a \emph{token} on vertex $u$. The objective in multicover is to minimize the sum of the $y_u$ such that each ball of radius $k$ has at least $k$ tokens on its vertices.  In symbols,

\begin{equation}\label{eqn:Mc}
\begin{matrix}
M_c(G) & = & \displaystyle \min \sum_{u \in V} y_u & \\[20pt]
  & s.t. & \displaystyle \sum_{d(u,v) \leq k} y_u \geq k & \mbox{ for each } v \in V \mbox{ and } 1 \leq k \leq e_c(v)\\[20pt]
   & & y_u \in \{ 0, 1 \}
\end{matrix}   
\end{equation}

As these are dual programs, we have $P_b(G) \leq M_c(G)$ (and equality in the case that we consider the fractional relaxation of both programs).  As a result, if we can find a feasible broadcast packing and a feasible multicover of the same weight, then both must be optimal. By the work in~\cite{Farber1984,Lubiw1987} these programs have integer optima (with equal objective functions) for strongly chordal graphs. In particular, trees are strongly chordal so we can compute efficiently $P_b(T)=M_c(T)$ for any tree $T$.  Below we go further and give actual formulas for special classes of trees.

Before studying specific classes, we make the following observation used for verifying the multicover constraints of Eq.~\eqref{eqn:Mc}.

\begin{obs}
Let $T$ be a rooted tree where all leaves are on the same level. Suppose $l$ is a leaf and $v$ is an ancestor of $l$. Then for any $r \geq d(l,v)$, $N_r[l] \subseteq N_r[v]$.
\end{obs}

\subsection{Perfect binary trees}

We begin by describing an optimal multicover for the perfect binary tree $T_h$.

\begin{thm}\label{thm:multicover}
Let $T_{h}$ be a perfect binary tree of height $h = 3m+t \geq 4$ where $0 \leq t \leq 2$. Define $S\subseteq V(T_{h})$ by:
\begin{enumerate}
    \item all vertices at height $h-1$ belong to $S$,
    \item all vertices at height $h-2, h-5, \dots, 1+t$ belong to $S$,
    \item the root belongs to $S$ when $t=0$ or $t=1$, whereas the two vertices at level $1$ belong to $S$ when $t=2$.
\end{enumerate}
The set $S$ is a multicover of $T_h$.
\end{thm}

\begin{proof}
First we will show that the multicover constraints in Eq.\eqref{eqn:Mc} are satisfied for leaves. Let $l$ be a leaf and consider $N_r[l]$. Since all vertices on levels $h-1$ and $h-2$ have tokens, it is clear that $N_r[l]$ has $r$ tokens for $1 \leq r \leq 4$. Assume $r \geq 5$. Let $v$ be the ancestor of $l$ at distance $\lceil \frac{r}{2} \rceil$ and let $T'$ be the subtree rooted at $v$. The tokens on levels $h-1$ and level $h-2$ (of $T$) show $T'$ contains at least $2^{\lceil \frac{r}{2} \rceil -1} + 2^{\lceil \frac{r}{2} \rceil -2} \geq r$ tokens. As $V(T') \subseteq N_r[l]$, the multicover constraints hold for all leaves.

Now consider an internal vertex $v$, different from the root, at distance $d$ from the nearest leaf $l$.  For any $r \geq d$, we have $N_r[l] \subseteq N_r[v]$.  We have already established $N_r[l]$ contains at least $r$ tokens.  Thus consider $1 \leq r < d$.  If $r \geq 4$, then the construction of $S$ ensures that in the subtree rooted at $v$ all the vertices at one of the distances $r, r-1,$ or $r-2$ from $v$ have tokens.  Since $2^{r-2} \geq r$, the neighbourhood $N_r[v]$ contains $r$ tokens.  For $r=1$, there is never more than two consecutive levels without tokens, so $v$, its parent, or its children have tokens.  For $r=2$, $N_2[v]$ has 2 vertices at the same level as $v$ (including $v$), two children of $v$, and $4$ grandchildren of $v$ (recall $r < d$).  One of these three levels has at least 2 tokens.  For $r=3$, let $v$ be at level $i$.  Then in $N_3[v]$ there are at least $8$ vertices at level $i+3$, $4$ at level $i+2$, $4$ at level $i+1$.  At least one of these levels has tokens.

When $v$ is the root, a similar analysis shows $N_r[v]$ contains $r$ tokens.  For $r \geq 4$, we use $2^{r-2} \geq r$ as above.  For $1 \leq r \leq 3$, a case analysis based on $h \bmod{3}$ verifies the multicover conditions are met.
\end{proof}

Summing the number of elements in $S$ from above and recalling $M_c(G)$ is the weight of a minimum multicover, we obtain the following proposition.

\begin{prop}\label{prop:mcthupperbnd}
For $T_h$, $h = 3m+t \geq 4$, $0 \leq t \leq 2$, we have
\[
M_c(T_h) \leq 2^{t+1} \frac{8^{m}-1}{7} + 2^{h-1} + b_t
\]
where $b_0 = 1, b_1 = 1, b_2 = 2$.
\end{prop}

Next we define a packing broadcast for $T_h$ with the same weight as $M_c(T_h)$.  By duality of linear program this proves the optimality of both the multicover and the packing.

Consider the vertices at height $i$ and label them $v_1, v_2, \dots, v_{2^i}$ so that $v_{2j-1}$ and $v_{2j}$ have a common parent.  In this section, to \emph{broadcast at half the vertices at height $i$} means broadcast with power $1$ at vertices $v_1, v_3, \dots, v_{2^i-1}$ and do not broadcast at vertices $v_2, v_4, \dots, v_{2^i}$. 

\begin{thm}
For the perfect binary tree $T_h$, $h = 3m+t \geq 4$, $0 \leq t \leq 2$,
\[
P_b(T_h) = M_c(T_h) =  2^{t+1} \frac{8^{m}-1}{7} + 2^{h-1} + b_t
\]
where $b_0 = 1, b_1 = 1, b_2 = 2$.
\end{thm}

\begin{proof}

We define a broadcast $f$ as follows. For each vertex $v$ at height $h-2$, let $l_1, l_2, l_3, l_4$ be the four leaves in the subtree rooted at $v$ with $l_1, l_2$ sharing a parent, and $l_3, l_4$ sharing a parent.  Set $f(l_1) = 2$, $f(l_3) = 1$ and $f(l_2) = f(l_4) = 0 $.

For heights $h-4, h-7, \dots, c_t$ where $c_0 = 5, c_1 = 3,$ and $c_2 = 4$, $f$ broadcasts at half the vertices at each of these heights. Finally, $f$ broadcasts with power $p_t$ at the root where $p_0 = 3, p_1 = 1,$ and $p_2 = 2$. It is easy to verify that $f$ is a packing broadcast. 

The weight of $f$ is determined as follows. Let $v$ be a vertex at height $h-2$.  In the subtree rooted at $v$, there is a leaf of weight $2$ and a leaf of weight $1$ giving a total weight for all such vertices of $2^{h-1}+ 2^{h-2}$. There are broadcast vertices of weight $1$ at half the vertices on levels $h-4, h-7, \dots, c_t$ giving a total weight of $2^{h-5}+2^{h-8} + \cdots + 2^{c_t-1}$.  Finally the root broadcasts with power $p_t$.  It is straightforward to simplify these sums to:
\[
    \begin{dcases}
        2 \cdot \frac{8^{m}-1}{7} + 2^{h-1} + 1 & t=0 \\
        4 \cdot \frac{8^{m}-1}{7} + 2^{h-1} + 1 & t=1 \\
        \frac{8^{m+1}-1}{7} + 2^{h-1} + 1 & t=2. \\
    \end{dcases}
\]
We have a multicover and a packing broadcast with the same total costs. The result follows.
\end{proof}

\subsection{Perfect $k$-ary trees}

The following result is the $k$-ary analogue of Theorem~\ref{thm:multicover}. (Recall we required $k \geq 3$.) We omit the proof as it is similar to the proof of Theorem~\ref{thm:multicover}.

\begin{thm}\label{thm:karymulticover}
   Let $T^k_{h}$ be the perfect $k$-ary tree of height $h =3m + t \geq 3$, $0 \leq t \leq 2$. Define the set $S \subseteq V(T_h^k)$  as follows: 
   \begin{enumerate}
    \item all vertices at height $h-1$ belong to $S$.
    \item all vertices at height $h-2, h-5, \dots, 1+t$
    \item the root belongs to $S$ when $t=1$, whereas the $k$ vertices at level $1$ belong to $S$ when $t=2$
\end{enumerate}
The set $S$ is a multicover of $T_h^k$.
\end{thm}

Computing the size of $S$ in Theorem~\ref{thm:karymulticover}, we obtain the immediate bound.

\begin{cor}\label{mckarybound}
    For $T_h^k$, $h=3m+t \geq 4$, $0 \leq t \leq 2$, we have
    \[
    M_c(T_h^k) \leq k^{t+1} \frac{(k^3)^{m-1 }-1}{k^3-1} + 2\cdot k^{h-2} + (k-1) \cdot k^{h-2} + b_t
    \]
where $b_0 = 0, b_1 = 1, b_2 = 1$.
\end{cor}

Next we define a packing broadcast for $T_h^k$ with the same weight as the upper bound for $M_c(T_h^k)$ in Corollary~\ref{mckarybound}. By duality of linear program this proves the optimality of both the multicover and the packing.

\begin{thm}
For $T_h^k$, $h=3m+t \geq 4$, $0 \leq t \leq 2$, we have
    \[
    P_b(T_h^k) = M_c(T_h^k) = k^{t+1} \frac{(k^3)^{m-1 }-1}{k^3-1} + 2\cdot k^{h-2} + (k-1) \cdot k^{h-2} + b_t
    \]
where $b_0 = 0, b_1 = 1, b_2 = 1$.
\end{thm}
\begin{proof}
We define a broadcast $f$ as follows. First, we examine the lowest levels. Let $v$ be a vertex at height $h-2$.  Let $l_1, l_2, \dots, l_{k^2}$ be the $k^2$ leaves in the subtree rooted at $v$ with $l_1, l_2, \dots, l_k$ sharing a parent, $l_{k+1}, l_{k+2} , \dots, l_{2k}$ sharing a parent, $\dots$,  and $l_{k^2 - k +1}, \dots, l_{k^2}$ sharing a parent. Broadcast with power $2$ at $l_1$ and power $1$ at $l_{k+1}, l_{2k+1}, \dots,$ $l_{k^2 - k +1}$.  All other leaves in this subtree do not broadcast.

Consider the vertices at height $i$ and label them $v_1, v_2, \dots, v_{k^i}$ so that $v_{(j-1) k+1}, \dots, v_{kj}$ have a common parent.  To \emph{broadcast at $\frac{1}{k}^{th}$ of the vertices at height $i$} means broadcast at vertices $v_1, v_{k+1}, \dots, v_{k^i-k+1}$ with power $1$. Let $f$ broadcast at $\frac{1}{k}^{th}$ of the vertices at height $h-4, h-7, \dots, h_t$ where $h_0 = 2, h_1 = 3,$ and $h_2 = 4$.  Finally broadcast with power $p_t$ at the root where $p_0 = 0, p_1 = 1,$ and $p_2 = 2$.

It is straightforward to verify $f$ is a packing broadcast, and that the weight of $f$ matches the value given in the theorem statement.

    We have given both a multicover and a packing broadcast with the same total cost. The result follows. 
\end{proof}

\subsection{Spiders and caterpillars}\label{sec:mcspidercat}

In this section, we study spiders, as defined in Section \ref{sec:spiders}. Recall a spider is denoted by $S(d_1, \dots, d_k)$, with a branch point denoted by $u$. Note that if at most two of the $d_i$ are greater than $1$, then $S(d_1, \dots, d_k)$ is a caterpillar.

\begin{thm}
    Let $T = S(d_1, \dots, d_k)$ be a spider, that is not a caterpillar. Let $S \subseteq  V(T)$ consist of all internal vertices of T. Then S is a multicover of $T$.
    \begin{proof}
        Consider $v \in V(T)$ and any ball $N_r[v]$, where $r \leq \mathrm{ecc}(v)$. Let $l_i$ be a leaf such that $d(l_i,v) = \mathrm{ecc}(v)$. Let $w$ be the vertex at distance $r$ on the unique $vl_i-$path. If at least one of $v$ or $w$ is not a leaf, there are at least $r$ internal vertices of $T$ on the $vw-$path, hence at least $r$ tokens within $N_r[v]$. If both $v$ and $w$ are both leaves, then $u$ is on the $vw-$path. The $vw-$path contains $r-1$ internal vertices of $T$. Since $d_i \geq 2$ for at least three values of $i$, there is a token on a vertex adjacent to $u$ not on the $vw-$path. Therefore, $N_r[v]$ has at least $r$ tokens. Therefore, $S$ is a multicover. It is easy to verify 
        \[ 
            \vert S \vert = \sum_{m=1}^{k} (d_m-1) + 1. 
        \] 
    \end{proof}
\end{thm}

\begin{cor}
Let $T = S(d_1, \dots, d_k)$ be a spider, that is not a caterpillar. Then 
\[ 
    P_b(T) = M_c(T) = \sum_{m=1}^{k} (d_m-1) + 1. 
\]
    \begin{proof}
        Consider the broadcast $f$ such that $f(l_1) = d_1$, $f(l_i) = d_i-1$ for all $2 \leq i \leq k$ and $f(v) = 0$ for all other vertices of $T$. It is easy to verify that $f$ is a packing broadcast with cost         
        $\sum_{m=1}^{k} (d_m-1) + 1$. The result follows.
    \end{proof}
\end{cor}

\begin{thm}
    Let $T$ be a caterpillar. Then $P_b(T)=\mathrm{diam}(T)$.
    \begin{proof}
        In~\cite{Dunbar2006}, the authors state that $P_b(T) \geq \mathrm{diam}(T)$. We show $P_b(T) \leq \mathrm{diam}(T)$. Let $P$ be a diametrical path in T, and let $v$ be an endpoint of P. Let $S = V(P) - \{ v\}$. So $|S| = \mathrm{diam}(T)$. It is easy to verify $S$ is a multicover of $T$. Therefore, we have $P_b(T) \leq M_C(T) \leq \mathrm{diam}(T)$, as required.
    \end{proof}
\end{thm}

\subsection{Double spiders}\label{sec:mcdoublespiders}

Let $S_1 = S(a_1, \dots, a_k)$ and $S_2 = S(b_1, \dots, b_l)$ be spiders with branch vertices $v_1$ and $v_2$ respectively. Let $d$ be a positive integer. We define a \emph{double spider} $DS = (S_1, S_2, d)$ to be the graph obtained by joining disjoint copies of $S_1$ and $S_2$ with a path of length $d$ between $v_1$ and $v_2$. (For this section we relax the definition of spider to allow $k=2$ or $l=2$ so that the degree $v_1$ and $v_2$ are at least $3$ in the double spider.) Figure \ref{fig:great} shows the double spider $DS = (S_1, S_2, 5)$, where $S_1 = (1,2,2)$ and $S_2 = (3,3,3)$. Note that double spiders are trees, thus internal vertices and leaves are well-defined for double spiders. The set of leaves whose branch vertex is $v_i$ is denoted $L(v_i)$, for $i = 1,2$. Define $l_1$, respectively $l_2$, be the leaf at maximum distance from its branch vertex $v_1$, respectively $v_2$.

Let $T$ be a double spider with branch vertices $v_1$ and $v_2$. We will show that the following algorithm creates a multicover of $T$.  We first define several variables. Let $m_1 = d(l_1,v_1)$, $m_2 = d(l_2,v_2)$, $t_1 = \sum_{l \in L(v_1)} d(l,v_1) - \vert L(v_1) \vert +1$ and $t_2 = \sum_{l \in L(v_2)} d(l,v_2) - \vert L(v_2) \vert +1$. Assume with loss of generality, that $t_1 - m_1 \leq t_2 - m_2$. Label the set of vertices along the $v_1v_2$-path as $v_1,p_1,p_2, \dots, p_{d-1},v_2$. 

\begin{enumerate}
  \item Put a token on $v_1$ and $v_2$. For each $l \in L(v_1)$, put a token on the internal vertices of the unique $lv_1$-path.  Repeat the process for each $l \in L(v_2)$.
  \item Place tokens on $p_2, p_4, \dots p_{d-2} $ or $p_{d-3}$ if $d$ is even or odd, respectively.
  \item If $2(t_1 - m_1) > d-1$, we are done. If $2(t_1 - m_1) = d-1$, place a token on $p_{2(t_1 - m_1)}$ and then we are done. Otherwise $2(t_1-m_1) < d-1$, place tokens on $p_{2(t_1 - m_1)+1}, p_{2(t_1 - m_1)+3}, \dots, p_{d-2}$ or $p_{d-1}$ if $d$ is odd or even respectively.
\end{enumerate}

\begin{figure}[hb]
    \begin{center}
    \begin{tikzpicture}[scale=0.9]
        \node[vertex,label={180:$7$}] (v0) at (0,0){};    
        \node[vertex] (v1) at (1,0){};
        \node[vertex] (v2) at (2,0){};         
        \node[vertex] (v3) at (3,0){};
        \node[vertex] (v4) at (4,0){};
        \node[vertex] (v5) at (5,0){};        
        \node[vertex] (v6) at (6,0){};
        \node[vertex] (v7) at (7,0){};
        \node[vertex] (v8) at (8,0){};
        \node[vertex] (v9) at (9,0){};
        \node[vertex, label={360:$2$}] (v10) at (10,0){};
        \node[vertex] (v11) at (1,0.5){};
        \node[vertex] (v12) at (0,1){};
        \node[vertex] (v13) at (1,-0.5){};
        \node[vertex] (v14) at (8,0.5){};
        \node[vertex] (v15) at (9,1){};
        \node[vertex, label={360:$2$}] (v16) at (10,1.5){};
        \node[vertex] (v17) at (8,-0.5){};
        \node[vertex] (v18) at (9,-1){};
        \node[vertex, label={360:$2$}] (v19) at (10,-1.5){};
    
        \node[token] (t1) at (1,0){};
        \node[token] (t2) at (2,0){};         
        \node[token] (t4) at (4,0){};
        \node[token] (t5) at (5,0){};        
        \node[token] (t6) at (6,0){};
        \node[token] (t7) at (7,0){};
        \node[token] (t8) at (8,0){};
        \node[token] (t9) at (9,0){};
        \node[token] (t11) at (1,0.5){};
        \node[token] (t14) at (8,0.5){};
        \node[token] (t15) at (9,1){};
        \node[token] (t17) at (8,-0.5){};
        \node[token] (t18) at (9,-1){};

        \draw[thick] (v0)--(v1)--(v2)--(v3)--(v4)--(v5)--(v6)--(v7)--(v8)--(v9)--(v10);
        \draw[thick](v2)--(v11)--(v12);
        \draw[thick](v2)--(v13);
        \draw[thick](v7)--(v14)--(v15)--(v16);
        \draw[thick](v7)--(v17)--(v18)--(v19);

    \end{tikzpicture}
    \caption{A packing broadcast and multicover of a double spider where $2(t_1 - m_1) < d-1$.
    The multipacking is denoted by the squares and the broadcast is indicated by the value next to the broadcasting leaves. Here $m_1 = 2, m_2 = 3, t_1 = 3, t_2 = 7$ and $d = 5$.}    \label{fig:less}
\end{center}
\end{figure}
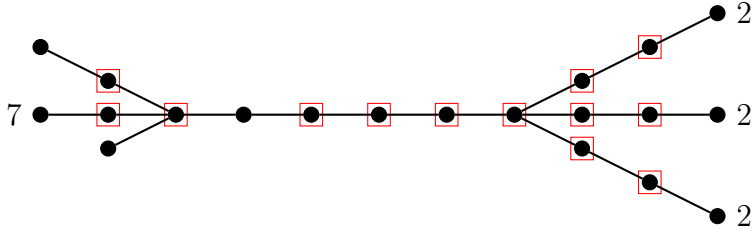

\begin{figure}[ht]
    \begin{center}
    \begin{tikzpicture}[scale=0.9]
        \node[vertex, label={180:$2$}] (v0) at (0,0){};    
        \node[vertex] (v1) at (1,0){};
        \node[vertex] (v2) at (2,0){};         
        \node[vertex] (v3) at (3,0){};
        \node[vertex, label={90:$1$}] (v4) at (4,0){};
        \node[vertex] (v5) at (5,0){};        
        \node[vertex] (v6) at (6,0){};
        \node[vertex] (v7) at (7,0){};
        \node[vertex] (v8) at (8,0){};
        \node[vertex, label={360:$3$}] (v9) at (9,0){};
        \node[vertex] (v10) at (1,0.5){};
        \node[vertex, label={180:$1$}] (v11) at (0,1){};
        \node[vertex] (v12) at (1,-0.5){};
        \node[vertex, label={180:$1$}] (v13) at (0,-1){};
        \node[vertex] (v14) at (7,0.5){};
        \node[vertex] (v15) at (8,1){};
        \node[vertex, label={360:$2$}] (v16) at (9,1.5){};
        \node[vertex] (v17) at (7,-0.5){};
        \node[vertex] (v18) at (8,-1){};
        \node[vertex, label={360:$2$}] (v19) at (9,-1.5){};

        \node[token] (t1) at (1,0){};
        \node[token] (t2) at (2,0){};
        \node[token] (t4) at (4,0){};
        \node[token] (t5) at (6,0){};        
        \node[token] (t6) at (7,0){};
        \node[token] (t7) at (8,0){};
        \node[token] (t10) at (1,0.5){};
        \node[token] (t12) at (1,-0.5){};
        \node[token] (t15) at (7,0.5){};
        \node[token] (t15) at (8,1){};
        \node[token] (t17) at (7,-0.5){};
        \node[token] (t18) at (8,-1){};
        
        \draw[thick] (v0)--(v1)--(v2)--(v3)--(v4)--(v5)--(v6)--(v7)--(v8)--(v9);
        \draw[thick] (v2)--(v10)--(v11);
        \draw[thick] (v2)--(v12)--(v13);
        \draw[thick] (v6)--(v14)--(v15)--(v16);
        \draw[thick] (v6)--(v17)--(v18)--(v19);
        
    \end{tikzpicture}
    \caption{A packing broadcast and multicover of a double spider where $2(t_1 - m_1) > d-1$.  The multipacking is denoted by the squares and the broadcast is indicated by the value next to the broadcasting leaves and $p_2$. Here $m_1 = 2, m_2 = 3, t_1 = 4, t_2 = 7$ and $d = 4$.}    \label{fig:great}
\end{center}
\end{figure}
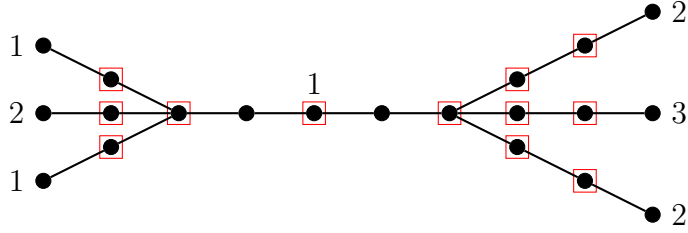

We now give a series of propositions which prove that the above algorithm creates a multicover of $T$. For the following, define $v_1, v_2, l_1, l_2, t_1,$ $t_2, m_1, m_2$ and $p_1, p_2, \dots, p_{d-1}$ as above and assume $t_1 - m_1 \leq t_2 - m_2$. We comment that $t_i$ is the number of tokens on the spider $S_i$, a fact we use below.

\begin{prop}\label{prop:dsleaves}
    For any leaf $l$, $N_r[l]$ contains at least $r$ tokens.
    \begin{proof}
       We claim that we need only show that $N_r[l_1]$ and $N_r[l_2]$ both contain at least $r$ tokens. Let $l$ be a leaf such that $v_1$ is the branch vertex for $l$, and $l \ne l_1$. By the above process, for $1 \leq r \leq d(l_1,l)-1$, $N_r[l]$ contains $r$ tokens. Now consider $d(l_1,l) \leq r \leq \mathrm{ecc}(l)$. We will show $N_r[l_1] \subseteq N_r[l]$. Let $v \in N_r[l_1]$. If $v$ lies on the unique $l_1l$-path, $d(l,v) \leq d(l, l_1) \leq r$. Otherwise, both the unique $l_1v$-path and $lv$-path pass through $v_1$. Then $d(l,v) \leq d(l_1,v) \leq r$. So $N_r[l_1] \subseteq N_r[l]$ as required. The same argument follows for $l_2$. This establishes the claim.
       
       We now show $N_r[l_1]$ has at least $r$ tokens. Let $l_1'$ be the leaf whose branch vertex is $v_1$ and $d(l_1',v_1)$ is maximum over all $l_1' \in L(v_1)$, $l_1' \ne l_1$ (so, $d(l_1',v_1)$ is the length of the second longest branch in $S_1$). Let $m_1' = d(l_1',v_1)$. Consider the following cases:

        \emph{Case $1$: $1\leq r < m_1+m_1'$:} As the $l_1 l_1'$-path contains a token on every vertex excluding $l_1$ and $l_1'$, clearly $N_r[l_1]$ contains $r$ tokens.

        \emph{Case $2$: $m_1 + m_1' \leq r \leq m_1 + 2(t_1-m_1)$:} Since $r \geq m_1 + m_1'$, the ball $N_r[l_1]$ contains all of $S_1$ and thus there are at least $t_1$ tokens within $N_r[l_1]$. By the algorithm, there are at least $\lfloor \frac{r-m_1}{2} \rfloor$ additional tokens along $P$ belonging to $N_r[l_1]$. Since, 
        \[ 
        t_1 +  \left\lfloor \frac{r-m_1}{2} \right\rfloor \geq t_1 + \frac{r-m_1}{2} -\frac{1}{2} 
        = \frac{m_1 + 2(t_1-m_1)}{2} + \frac{r}{2} - \frac{1}{2} \geq r - \frac{1}{2}
        \]
        and the number of tokens is an integer, there are at least $r$ tokens with $N_r[l_1]$. 

        \emph{Case $3$: $m_1 + 2(t_1 - m_1) < r$:} Consider the vertex at distance $r$ from $l_1$, say $u$. Since $m_1 \leq t_1$, the ball $N_r[l_1]$ contains all of $S_1$ and thus contains $t_1$ tokens. 
        
        Let $Q$ be the $l_1u$-path which we label as $l_1, w_1, \dots, v_1, \dots, p_{2(t_1-m_1)},$ $p_{2(t_1-m_1)+1},$ $\dots, x, u$, where $u = p_j$ for some $j \geq 2(t_1-m_1)+1$ or $u \in S_2$. We now count the tokens in $N_r[l_1]$ on this path that are not in $S_1$.  There tokens on every other vertex of $p_1, \dots, p_{2(t_1-m_1)}$, giving $2(t_1-m_1)/2$ tokens. 
        
        There are tokens on each vertex of $p_{2(t_1-m_1)+1}, \dots, u$ if $u \not\in L(v_2)$ or each vertex of $p_{2(t_1-m_1)+1}, \dots, x$ if $u \in L(v_2)$.  Since $T$ is not a caterpillar we cannot have $t_2 - m_2 = t_1 - m_1 = 0$. By assumption $t_2-m_2 \geq t_1 - m_1$, so $t_2 > m_2$.  Thus in the case $u \in L(v_2)$, there is a neighbour of $v_2$ not on $Q$ but containing a token and belonging to $N_r[l_1]$.  The path from $p_{2(t_1-m_1)+1}$ to $x$ has length $r-2(t_1-m_1)-m_1$ and has a token on each vertex.  There is a final token either on $u$ or on a neighbour of $v_2$ (not belonging to $Q$) as described above. 

        In all cases we have $N_r[l_1]$ contains at least
        \[
            t_1 + \frac{2(t_1-m_1)}{2} + r - 2(t_1-m_1) - m_1 \\
            = r
        \]
        tokens as required.

        So for all $r \leq \mathrm{ecc}(l_1)$, $N_r[l_1]$ contains at least $r$ tokens. A similar case analysis shows that for all $r$, $N_r[l_2]$ contains at least $r$ tokens.
    \end{proof}
\end{prop}

\begin{prop}\label{prop:int1}
    Let $v$ be an internal vertex of $T$. If $N_r[v]$ contains a path of length $2r$, then $N_r[v]$ contains at least $r$ tokens.
    \begin{proof}
        Let $P$ be the path of length $2r$ contained in $N_r[v]$. By the algorithm, there is at most one pair of adjacent vertices $u$ and $w$ which both do not contain a token (this occurs when $d-1$ is odd and $2(t_1-m_1) > d-1$, giving $ u = p_{d-2}$ and $w = p_{d-1}$). If $N_r[v]$ does not contain both $u$ and $w$, then $N_r[v]$ has a token on at least every second vertex of $P$. So there are at least $ \lfloor \frac{2r+1}{2} \rfloor = r$ tokens contained in $N_r[v]$. Now suppose $w,u \in N_r[v]$. Consider the path $P'$ of length $2r-1$ in $N_r[v]$ made by contracting the edge $uw$. Since $u,w$ was the only pair of adjacent vertices which both did not contain a token, $P'$ must contain at least $\lfloor \frac{2r}{2} \rfloor = r$ tokens. Thus $N_r[v]$ contains at least $r$ tokens. 
    \end{proof}
\end{prop}

\begin{prop}\label{prop:int2}
    Let $v$ be an internal vertex of $T$. If $N_r[v]$ does not contain a path of length $2r$, then $N_r[v]$ contains at least $r$ tokens.
    \begin{proof}
        Let $w$ be a vertex at distance $r$ from $v$ and consider a longest path starting at $w$ which contains $v$. Note that the other endpoint of this path must be a leaf, say $l$. We claim $N_r[l] \subseteq N_r[v]$. Let $u \in N_r[l]$.  Either $v$ belongs to the $lu$-path or $v$ belongs to the $uw$-path (otherwise the tree has a $lw$-walk that avoids $v$ which is impossible).  In the former case we have $d(l,v)+d(v,u) = d(l,u) \leq r$ and in the latter we have $d(w,u)= d(w,v)+d(v,u) = r + d(v,u) < 2r$.  In both cases $d(v,u) \leq r$ giving $u \in N_r[v]$.

        So $N_r[l] \subseteq N_r[v]$ as required. Thus by Proposition \ref{prop:dsleaves}, $N_r[v]$ contains at least $r$ tokens.
    \end{proof}
\end{prop}

\begin{cor}\label{cor:mcds}
The above algorithm creates a multicover of a double spider $T = DS(S_1, S_2, d)$.
\begin{proof}
    The result follows from Propositions \ref{prop:dsleaves}, \ref{prop:int1} and \ref{prop:int2}.
\end{proof}
\end{cor}

\begin{thm}\label{thm:pbds}
    Let $T=DS(S_1, S_2, d)$ be a double spider with $m_1, t_1, t_2$ defined as above.  Then 
    \[M_c(T) = P_b(T)=
        \begin{dcases}
            m_1+d+t_2-1,  & \text{ if } 2(t_1-m_1) \leq d-1, \\
            t_1+\lfloor {(d-2)}/{2}\rfloor+t_2, & \text{ if } 2(t_1-m_1) \geq d.
        \end{dcases}
    \]
    \begin{proof}
        If $2(t_1-m_1) \leq d-1$, define a broadcast $f$ as 
        \[f(x)=
            \begin{dcases}
                d(l_1,v_2), & \text{ if } x = l_1, \\
                d(x,v_2) - 1, & \text{ if } x \in L(v_2), \\
                0, & \text{ otherwise.}
            \end{dcases}
        \]
        It is easy to verify $f$ is a packing broadcast with the same size as the mutlicover described above.

        If $2(t_1-m_1) \geq d$, define a broadcast $f$ as 
        \[ f(x)=
            \begin{dcases}
                d(l_1,v_1), & \text{ if } x = l_1, \\
                d(l_2,v_2), & \text{ if } x = l_2, \\
                d(x,v_1) - 1, & \text{ if } x \in L(v_1) - \{ l_1\}, \\
                d(x,v_2) - 1, & \text{ if } x \in L(v_2) - \{ l_2\}, \\
                \lfloor {(d-2)}/{2} \rfloor, & \text{ if } x = p_{\lfloor \frac{d}{2} \rfloor}, \\
                0, & \text{ otherwise.}
            \end{dcases}
        \]
        It is easy to verify $f$ is a packing broadcast with the same size as the mutlicover described above.
        The result follows.
    \end{proof}
\end{thm}

\section{Further questions}

This work began with a question of Ahmane, Bouchemakh, and Sopena~\cite{Ahmane2018} on finding broadcast independence values for other classes of trees, in particular $k$-ary trees.  The most natural question is what other classes of graphs, or even trees, can we compute the broadcast independence number or broadcast packing number?

\noindent\textbf{Acknowledgements\hspace{0.1in}}We acknowledge the support of
the Natural Sciences and Engineering Research Council of Canada (NSERC), RGPIN-2014-04760 and the USRA program.

\noindent Cette recherche a \'{e}t\'{e} financ\'{e}e par le Conseil de
recherches en sciences naturelles et en g\'{e}nie du Canada (CRSNG), RGPIN-2014-04760 et le programme BRPC.

\begin{center}
\includegraphics[width=2.5cm]{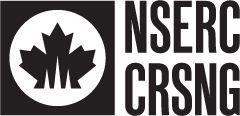}
\end{center}

\bibliography{broadcastlit}{}
\bibliographystyle{abbrv}

\end{document}